\documentclass{amsart}
\usepackage[dvips]{graphicx}
\usepackage{amscd}
\usepackage{pstricks}
\usepackage{amssymb}
\theoremstyle{plain}
\newtheorem{theorem}{Theorem}
\newtheorem{proposition}{Proposition}

\newtheorem{definition}{Definition}
\newtheorem{corollary}{Corollary}

\theoremstyle{remark}
\newtheorem{remark}{Remark}

\numberwithin{equation}{section}

\DeclareMathOperator{\tr}{Tr}

\begin{document}

\title[]{Differentiability of Palmer's linearization Theorem and converse result for density functions}
\author[]{\'Alvaro Casta\~neda}
\author[]{Gonzalo Robledo}
\address{Departamento de Matem\'aticas, Facultad de Ciencias, Universidad de
  Chile, Casilla 653, Santiago, Chile}
\email{castaneda@u.uchile.cl,grobledo@u.uchile.cl}
\thanks{The first author was funded by FONDECYT Iniciaci\'on Project 11121122 and by 
the Center of Dynamical Systems and Related Fields DySyRF (Anillo Project 1103, CONICYT). The second author
was funded by FONDECYT Regular Project 1120709}
\date{November 2014}

\begin{abstract}
We study differentiability properties in a particular case of the Palmer's linearization Theorem, which
states the existence of an homeomorphism $H$ between the solutions of a linear ODE system having exponential
dichotomy and a quasilinear system. Indeed, if the linear system is uniformly asymptotically stable, sufficient 
conditions ensuring that $H$ is a $C^{2}$ preserving orientation diffeomorphism are given. As an application,
we generalize a converse result of density functions for a nonlinear system in the nonautonomous case.
\end{abstract}

\maketitle

\section{Introduction}
The seminal paper of K.J. Palmer \cite{Palmer} provides sufficient conditions ensuring the 
topological equivalence between the solutions of the linear system
\begin{equation}
\label{lin}
y'=A(t)y,
\end{equation}
and the solutions of the quasilinear one
\begin{equation}
\label{no-lin}
x'=A(t)x+f(t,x), 
\end{equation}
where the bounded and continuous $n\times n$ matrix $A(t)$ and the continuous function $f\colon \mathbb{R}\times \mathbb{R}^{n}\to \mathbb{R}^{n}$ 
satisfy some technical conditions.

Roughly speaking, (\ref{lin}) and (\ref{no-lin}) are topologically equivalents
if there exists a map $H\colon \mathbb{R}\times \mathbb{R}^{n}\to \mathbb{R}^{n}$
such that $\nu \mapsto H(t,\nu)$ is an homeomorphism for any fixed $t$. In particular,
if $x(t)$ is a solution of (\ref{no-lin}), then $H[t,x(t)]$ is a solution of (\ref{lin}). 

To the best of our knowledge, there are no results concerning the differentiability of the map
$H$ and the purpose of this paper is to find sufficient conditions
ensuring that the map above is a preserving orientation diffeomorphism of class $C^{1}$ (Theorem \ref{anexo-palmer}) 
and $C^{2}$ (Theorem \ref{SegDer}), both under the assumption that (\ref{lin}) is uniformly asymptotically stable.

As an application of our results, we will construct a density 
function for the system (\ref{no-lin}) when $f(t,0)=0$  (Theorem \ref{ex-dens0}), generalizing a 
converse result in the autonomous case presented 
in \cite{Monzon-0}.

\subsection{Palmer's linearization Theorem}

We are interested in the particular case:
\begin{proposition}[Palmer \cite{Palmer}]
\label{difeo}
If the assumptions: 
\begin{itemize}
\item[\textbf{(H1)}] $|f(t,x)|\leq \mu$ for any $t\in \mathbb{R}$ and $x\in \mathbb{R}^{n}$.
\item[\textbf{(H2)}] $|f(t,x_{1})-f(t,x_{2})|\leq \gamma |x_{1}-x_{2}|$ for any $t\in \mathbb{R}$,
where $|\cdot|$ denotes a norm in $\mathbb{R}^{n}$.
\item[\textbf{(H3)}] There exist some constants $K\geq 1$ and $\alpha>0$ such that
the transition matrix $\Psi(t,s)=\Psi(t)\Psi^{-1}(s)$ of \textnormal{(\ref{lin})} verifies
\begin{equation}
\label{cota-exponencial}
||\Psi(t,s)||\leq Ke^{-\alpha(t-s)}, \quad \textnormal{for any} \quad t\geq s.
\end{equation}
\item[\textbf{(H4)}] The Lipschitz constant of $f$ verifies: 
\begin{equation}
\label{palmer}
\gamma \leq \alpha/4K,
\end{equation}
\end{itemize}
are satisfied, there exists a unique function 
$H\colon \mathbb{R}\times \mathbb{R}^{n}\to \mathbb{R}^{n}$ satisfying
\begin{itemize}
 \item[i)] $H(t,x)-x$ is bounded in $\mathbb{R}\times \mathbb{R}^{n}$, 
 \item[ii)] If $t\mapsto x(t)$ is a solution of \textnormal{(\ref{no-lin})}, then
 $H[t,x(t)]$ is a solution of \textnormal{(\ref{lin})}.
\end{itemize}
Morevoer, $H$ is continous in $\mathbb{R}\times \mathbb{R}^{n}$ and
$$
|H(t,x)-x|\leq 4K\mu\alpha^{-1}
$$
for any $(t,x)\in \mathbb{R}\times \mathbb{R}^{n}$. For each fixed $t$, $H(t,x)$ is a homeomorphism
of $\mathbb{R}^{n}$. $L(t,x)=H^{-1}(t,x)$ is continous in $\mathbb{R}\times \mathbb{R}^{n}$ and if $y(t)$ is any solution
of \textnormal{(\ref{lin})}, then $L[t,y(t)]$ is a solution of \textnormal{(\ref{no-lin})}.
\end{proposition}

\begin{remark}
The original Palmer's result assumes that (\ref{lin}) has an exponential dichotomy
property. The condition \textbf{(H3)} is a particular case considering the identity as projector,
which implies that (\ref{lin}) is exponentially stable at $+\infty$. In addition, let us 
recall that uniform asymptotical and exponential stability are equivalent in the linear case
(see \cite{Coppel} or Theorem 4.11 from \cite{Khalil}).
\end{remark}

This result has been extended and generalized in several directions \cite{Barreira}, \cite{Jiang}, \cite{Lopez}
\cite{Palmer-79}, \cite{Palmer-79-2}, \cite{Xia} but there are no results about the differentiability 
of $x\mapsto H(t,x)$. In this article we provides sufficient conditions, described in term of $\Psi(t,s)$,
$Df$ and $D^{2}f$, such that $H$ is a $C^{p}$ ($p=1,2$) preserving orientation diffeomorphism.

\subsection{Density functions}
 Let us consider the system
\begin{equation}
\label{generico}
z'=g(t,z) \quad \textnormal{with} \quad g(t,0)=0 \quad \textnormal{for any $t\in \mathbb{R}$},
\end{equation}
where $g\colon \mathbb{R}\times \mathbb{R}^{n}\to \mathbb{R}^{n}$ is such that the existence, uniqueness 
and unbounded continuation of the solutions is verified. 

\begin{definition}
\label{density}
A density function of \textnormal{(\ref{generico})} is a $C^{1}$ function 
$\rho\colon \mathbb{R}\times \mathbb{R}^{n}\setminus \{0\}\to [0,+\infty)$, integrable outside a ball 
centered at the origin that satisfies
\begin{displaymath}
\frac{\partial \rho(t,z)}{\partial t}+\triangledown \cdot [\rho(t,z)g(t,z)]>0 
\end{displaymath}
almost everywhere with respect to $\mathbb{R}^{n}$ and for every $t\in \mathbb{R}$, where
$$
\triangledown \cdot [\rho g] = \triangledown \rho \cdot g + \rho[\triangledown\cdot g], 
$$
and $\triangledown \rho$, $\triangledown \cdot g$ denote respectively the gradient of $\rho$ and divergence
of $g$.
\end{definition}

The density functions were introduced by Rantzer in 2001 \cite{Rantzer} in order to obtain sufficient conditions 
for almost global stability of autonomous systems, we refer  
to \cite{Angeli}, \cite{Dimarogonas}, \cite{Vasconcelos}, \cite{Loizou} and \cite{Meinsma} for a deeper discussion 
and applications. The extension to the nonautonomous case has been proved in \cite{Monzon}, \cite{Schlanbusch}:
\begin{proposition}[Theorem 4, \cite{Schlanbusch}]
Consider the system \textnormal{(\ref{generico})} such that $z=0$ is a locally stable equilibrium point.
If there exists a density function associated to \textnormal{(\ref{generico})}, then for every initial time $t_{0}$, the sets
of points that are not asymptotically attracted by the origin has zero Lebesgue measure. 
\end{proposition}

Converse results (\emph{i.e.}, global asymptotic stability implies the existence of a density function) were 
presented simultaneously by Rantzer \cite{Rantzer-2002} and Monz\'on \cite{Monzon-0} in the autonomous case by using different 
methods. In particular, in \cite{Monzon-0} the author constructs a density function associated to the system
\begin{equation}
\label{no-lin-PM}
z'=g(z) \quad \textnormal{with} \quad g(0)=0, 
\end{equation}
where $0$ is a globally asymptotically stable equlibrium and $g\colon \mathbb{R}^{n}\to \mathbb{R}^{n}$ is a $C^{2}$ function,
whose jacobian matrix at $z=0$ has eigenvalues with negative real part.

Such construction has two steps: i) As $u'=Dg(0)u$ is globally asymptotically stable, it is well known that
there exists a density function $\rho(z)$ (we refer to Proposition 1 from \cite{Rantzer} for details). ii)
A $C^{2}$ preserving orientation diffeomorphism $h\colon \mathbb{R}^{n}\to \mathbb{R}^{n}$ is constructed, such that
$\bar{\rho}(z)$ defined by
\begin{equation}
\label{DR}
\bar{\rho}(z)=\rho(h(z))\det Dh(z)
\end{equation}
is a density function for (\ref{no-lin-PM}).  

To the best of our knowledge, there are few converse results in the nonautonomous framework.  A first one was presented
by Monz\'on in 2006:
\begin{proposition}[Monz\'on, \cite{Monzon}]
\label{dens-lin}
If \textnormal{(\ref{lin})} is globally asymptotically stable, then there exists a $C^{1}$ density 
function $\rho\colon \mathbb{R}\times \mathbb{R}^{n}\setminus \{0\}\to [0,+\infty)$, associated to \textnormal{(\ref{lin})}. 
\end{proposition}

If we assume that
\begin{itemize}
\item[\textbf{(H5)}] The function $f$ satisfies $f(t,0)=0$ for any $t\in \mathbb{R}$, 
\end{itemize}
By Gronwall's inequality, combined with \textbf{(H4)}, it is easy to deduce that
any solution $t\mapsto \phi(t,t_{0},x_{0})$ of (\ref{no-lin}) passing throuhg $x_{0}$ at $t=t_{0}$
verifies
$$
|\phi(t,t_{0},x_{0})|\leq Ke^{\{-\alpha+K\gamma\}(t-t_{0})}|x_{0}|, \quad \textnormal{for} \quad t\geq t_{0}
$$
which implies the exponential stability of (\ref{no-lin}) at $t=+\infty$.

This property prompt us to extend the Monzon's converse result \cite{Monzon-0} to the nonlinear system
(\ref{no-lin}) by constructing a density function in the same way as in (\ref{DR}), where $\rho$ is defined
by Proposition \ref{dens-lin} and replacing $h$ by $H(t,x)$ defined in Proposition \ref{anexo-palmer}.

The paper is organized as follows. The section 2 states our main results and its proofs. The section 3 is devoted 
to prove the converse result and to construct a density function for (\ref{no-lin}) and an illustrative example is presented 
in section 5.

\section{Main Results}
As usual, given a matrix $M(t)\in M_{n}(\mathbb{R})$, its trace will be denoted by $\tr M(t)$ while its
determinant by $\det M(t)$, the identity matrix is denoted by $I$. 

The solution of (\ref{no-lin}) passing through $\xi$ at $t_{0}$ will be denoted 
by $\phi(t,t_{0},\xi)$. It will be interesting to consider the map $\xi\mapsto \phi(t,t_{0},\xi)$
and its properties. Indeed, if $f$ is $C^{1}$, it is well known (see \emph{e.g.} \cite[Chap. 2]{Coddington}) that 
$\partial \phi(t,t_{0},\xi)/\partial \xi=\phi_{\xi}(t,t_{0},\xi)$ satisfies the matrix differential equation
\begin{equation} 
\label{MDE1}
\left\{
\begin{array}{rcl}
\displaystyle \frac{d}{dt}\phi_{\xi}(t,t_{0},\xi)&=&\{A(t)+Df(t,\phi(t,t_{0},\xi))\}\phi_{\xi}(t,t_{0},\xi).\\\\
\phi_{\xi}(t_{0},t_{0},\xi)&=&I.
\end{array}\right.
\end{equation}

Moreover, it is proved that (see \emph{e.g.}, Theorem 4.1 from \cite[Ch.V]{Hartman}) if $f$ is $C^{r}$ with $r>1$, then the map
$\xi \mapsto \phi(t,t_{0},\xi)$ is also $C^{r}$. In particular, if $f$ is $C^{2}$, we can verify that
the second derivatives $\partial^{2}\phi(s,t_{0},\xi)/\partial \xi_{j}\partial\xi_{i}$  are solutions 
of the system of differential equations
\begin{equation}
\label{MDE15}
\left\{
\begin{array}{rcl}
\displaystyle \frac{d}{dt}\frac{\partial^{2}\phi}{\partial\xi_{j}\partial\xi_{i}}&=&\displaystyle
\{A(t)+Df(t,\phi)\}\frac{\partial^{2}\phi}{\partial\xi_{j}\partial\xi_{i}}
+D^{2}f(t,\phi)\frac{\partial\phi}{\partial\xi_{j}}\frac{\partial\phi}{\partial\xi_{i}} \\\\
\displaystyle \frac{\partial^{2}\phi}{\partial\xi_{j}\partial\xi_{i}}  &=&0,
\end{array}\right.
\end{equation}
with $\phi=\phi(t,t_{0},\xi)$, for any $i,j=1,\ldots,n$.

Now, let us introduce the following conditions
\begin{enumerate}
\item[\textbf{(D1)}] $f(\cdot,\cdot)$ is 
$C^{2}$  and, for any fixed $t$, its first derivative is such that
\begin{displaymath}
\int_{-\infty}^{t}||\Psi(t,r)Df(r,\phi(r,0,\xi))\Psi(r,t)||_{\infty}\,dr<1.
\end{displaymath}
\item[\textbf{(D2)}] For any fixed $t$, $A(t)$ and $Df(t,\phi(t,0,\xi))$ are such that
\begin{displaymath}
\liminf\limits_{s\to-\infty}-\int_{s}^{t}\hspace{-0.2cm}\tr A(r)\,dr>-\infty \hspace{0.2cm} \textnormal{and} \hspace{0.2cm}
\liminf\limits_{s\to-\infty}-\int_{s}^{t}\hspace{-0.2cm}\tr \{A(r)+Df(r,\phi(r,0,\xi))\}\,dr>-\infty.
\end{displaymath}
\item[\textbf{(D3)}] For any fixed $t$ and $i,j=1,\ldots,n$, the following limit exists
\begin{equation}
\label{D2}
\lim\limits_{s\to -\infty}\frac{\partial Z(s,x(t))}{\partial x_{j}(t)}e_{i},
\end{equation}
where $x(t)=(x_{1}(t),\ldots,x_{n}(t))=\phi(t,0,\xi)$, $e_{i}$ is the $i$--th component of the canonical
basis of $\mathbb{R}^{n}$ and $Z(s,x(t))$ is a fundamental matrix of the $x(t)$--parameter dependent system 
\begin{equation}
\label{L2}
z'=F(s,x(t))z
\end{equation}
satisfying $Z(t,x(t))=I$, where $F(r,x(t))$ is defined as follows
\begin{equation}
\label{L1}
F(r,x(t))=\Psi(t,r)Df(r,\phi(r,t,x(t)))\Psi(r,t),
\end{equation}
\end{enumerate}

\begin{remark}
\label{about-hyp}
We will see that the construction of the homeomorphism $H$ considers the behavior of
$\phi(t,0,\xi)$ for any $t\in (-\infty,\infty)$. In particular, to prove that $H$ is a $C^{2}$ preserving orientation diffeomorphism
will require to know the behavior on $(-\infty,t]$. Indeed, notice that:

\textbf{(D1)} is a technical assumption, introduced to ensure that the homeomorphism $H(t,x)$ 
stated in Proposition \ref{difeo} is a $C^{1}$ diffeomorphism. It is interesting to point out that,
by using a result of p. 11 of \cite{Coppel}, we know that $|H[s,\phi(s,0,\xi)]|\to +\infty$ when $s\to -\infty$,
this fact combined with statement ii) from Proposition \ref{anexo-palmer} implies that $|\phi(s,0,\xi)|\to +\infty$. In consequence, 
\textbf{(D1)} suggest that the asymptotic behavior of $Df(s,x)$ (when $|x|\to +\infty$ and $s\to -\infty$) ensures integrability.

Moreover, let us note that appears in some results about asymptotical equivalence (see \emph{e.g.}, \cite{Akhmet},\cite{Rab}).

\textbf{(D2)} is introduced in order to assure that $H$ is a preserving orientation diffeomorphism.
We emphasize that this asumption is related to Liouville's formula and is used in the asymptotic integration
literature (see \emph{e.g.}, \cite{Eastham}).

\textbf{(D3)} is introduced to ensure that $H$ is a $C^{2}$ diffeomorphism. Notice that, as stressed by Palmer 
\cite[p.757]{Palmer}, the uniqueness of the solution of (\ref{no-lin}) implies
the identity 
\begin{equation}
\label{identity-IC}
\phi(s,t,\phi(t,0,\xi))=\phi(s,0,\xi),
\end{equation}
and this fact combined with $x(t)=\phi(t,0,\xi)$ allows us to see that $F(r,x(t))$ is the same function
of \textbf{(D1)}.
\end{remark}

\begin{theorem}
\label{anexo-palmer} 
If \textnormal{\textbf{(H1)--(H4)}} and \textnormal{\textbf{(D1)--(D2)}} are satisfied, then, for 
any fixed $t$, the function $x\mapsto H(t,x)$ is a preserving orientation diffeomorphism.
In particular, if $t\mapsto x(t)$ is a solution of \textnormal{(\ref{no-lin})}, then, for any fixed $t$,
$x(t)\mapsto H[t,x(t)]$ is a preserving orientation diffeomorphism.
\end{theorem}

\begin{proof}
In order to make the proof self contained, we will recall some facts of the Palmer's proof \cite[Lemma 2]{Palmer} 
tailored for our purposes.

Firstly, let us consider the system 
$$
z'=A(t)z-f(t,\phi(t,\tau,\nu)),
$$
where $t\mapsto \phi(t,\tau,\nu)$ is the unique solution of (\ref{no-lin}) passing through $\nu$ at $t=\tau$. Moreover,
it is easy to prove that the unique bounded solution of the above system is given by
$$
\chi(t,(\tau,\nu))=-\int_{-\infty}^{t}\Psi(t,s)f(s,\phi(s,\tau,\nu))\,ds.
$$

The map $H$ is constructed as follows:
\begin{displaymath}
H(\tau,\nu)=\nu+\chi(\tau,(\tau,\nu))=\nu-\int_{-\infty}^{\tau}\Psi(\tau,s)f(s,\phi(s,\tau,\nu))\,ds.
\end{displaymath}

It will be essential to note that the particular case $(\tau,\nu)=(t,\phi(t,0,\xi))$, leads to
\begin{equation}
\label{lect1}
\chi(t,(t,\phi(t,0,\xi)))=-\int_{-\infty}^{t}\Psi(t,s)f(s,\phi(s,t,\phi(t,0,\xi)))\,ds.
\end{equation}

In addition, (\ref{identity-IC}) allows us to reinterpret (\ref{lect1}) as
\begin{equation}
\label{lect2}
\chi(t,(t,\phi(t,0,\xi)))=-\int_{-\infty}^{t}\Psi(t,s)f(s,\phi(s,0,\xi)))\,ds.
\end{equation}

In consequence, when $(\tau,\nu)=(t,\phi(t,0,\xi))$, we have: 
\begin{equation}
\label{hom-palm}
H[t,\phi(t,0,\xi)]=\phi(t,0,\xi)+\chi(t,(t,\phi(t,0,\xi)))
\end{equation}
and the reader can notice that the notation $H[\cdot,\cdot]$ is reserved to the case where
$H$ is defined on a solution of (\ref{no-lin}).

Having in mind the double representation (\ref{lect1})--(\ref{lect2}) 
of $\chi(t,(t,\phi(t,0,\xi)))$, the map $H[t,\phi(t,0,\xi)]$ can be written as 
\begin{equation}
\label{hom-palm1}
H[t,\phi(t,0,\xi)]=\phi(t,0,\xi)-\int_{-\infty}^{t}\Psi(t,s)f(s,\phi(s,0,\xi))\,ds,
\end{equation}
or 
\begin{equation}
\label{hom-palm2}
H[t,\phi(t,0,\xi)]=\phi(t,0,\xi)-\int_{-\infty}^{t}\Psi(t,s)f(s,\phi(s,t,\phi(t,0,\xi)))\,ds.
\end{equation}

The proof that $\phi(t,0,\xi)\mapsto H[t,\phi(t,0,\xi)]$ is an homeomorphism for any fixed $t$ is given by Palmer in
\cite[pp.756--757]{Palmer}. In addition, it is straightforward to verify that 
(\ref{hom-palm1}) is a solution of (\ref{lin}) passing through $H[0,\xi]$ at $t=0$. 

Turning now to the proof that $\phi(\tau,\nu)\mapsto H(\tau,\nu)$ is a 
preserving orientation diffeomorphism for any fixed $\tau$, we will only consider the case when $(\tau,\nu)=(t,\phi(t,0,\xi))$
by using (\ref{hom-palm2}). The general case can be proved analogously and is left to the reader.

The proof is decomposed in several steps.

\noindent\textit{Step 1:} Differentiability of the map $\phi(t,0,\xi)\mapsto H[t,\phi(t,0,\xi)]$.

Let us denote $x(t)=\phi(t,0,\xi)$ for any fixed $t$. By using the fact that $f$ is $C^{1}$ 
combined with (\ref{MDE1}) and
\begin{equation}
\label{MDE2}
\frac{d}{dt}\Psi(t,s)=A(t)\Psi(t,s) \quad \textnormal{and} \quad \frac{d}{ds}\Psi(t,s)=-\Psi(t,s)A(s),
\end{equation}
we can deduce that: 
\begin{equation}
\label{delicat}
\begin{array}{rcl}
\displaystyle \frac{\partial H[t,x(t)]}{\partial x(t)}&=& \displaystyle I-\int_{-\infty}^{t}\Psi(t,s)Df(s,\phi(s,t,x(t)))\frac{\partial \phi(s,t,x(t))}{\partial x(t)}\,ds\\\\
\displaystyle &=&\displaystyle  I-\int_{-\infty}^{t}\frac{d}{ds}\Big\{\Psi(t,s)\frac{\partial \phi(s,t,x(t))}{\partial x(t)}\Big\}\,ds\\\\
\displaystyle &=&\displaystyle \lim\limits_{s\to -\infty}\Psi(t,s)\frac{\partial \phi(s,t,x(t))}{\partial x(t)}.
\end{array}
\end{equation}

In consequence, the differentiability of $x(t)\mapsto H[t,x(t)]$ follows if and only if the limit above exists.

\noindent\textit{Step 2:} (\ref{delicat}) is well defined.

By (\ref{MDE1}), we know that $\partial \phi(s,t,x(t))/\partial x(t)$ is solution of the equation:
\begin{equation}
\label{perturbacion1}
\left\{\begin{array}{rcl}
Y'(s)&=&\{A(s)+Df(s,\phi(s,t,x(t)))\}Y(s)\\
Y(t)&=&I.
\end{array}\right.
\end{equation}


By (\ref{MDE2}),(\ref{identity-IC}) and (\ref{perturbacion1}), the reader can 
verify that $Z(s,x(t))=\Psi(t,s)\frac{\displaystyle \partial \phi(s,t,x(t))}{\displaystyle \partial x(t)}$ 
is solution of the $x(t)$--parameter dependent matrix differential equation
\begin{equation}
\label{modif1}
\left\{\begin{array}{rcl}
\displaystyle \frac{dZ}{ds}&=&\big\{\Psi(t,s)Df(s,\phi(s,0,\xi))\Psi(s,t)\big\}Z(s) \\\\
Z(t)&=&I.
\end{array}\right. 
\end{equation}

A well known result of sucessive approximations (see \emph{e.g.}, \cite{Adrianova},\cite{Bellman}) states that
\begin{displaymath}
Z(s,x(t))=I-
\int_{s}^{t}F(r,\xi)\,dr+\sum\limits_{k=2}^{+\infty}(-1)^{k}\Bigg(\int_{s}^{t}F(r_{1},\xi)\,dr_{1}\cdots\int_{r_{k-1}}^{t}F(r_{k},\xi)\,dr_{k}\Bigg), 
\end{displaymath}
where $F(r,\xi)$ is defined by (\ref{L1}). Moreover, we also know that
$$
||Z(s,x(t))||\leq \exp\Big(\Big|\int_{s}^{t}||F(r,\xi)||\,dr\Big|\Big)
$$
and \textbf{(D1)} implies that (\ref{delicat}) is well defined. 

\noindent \emph{Step 3:} $x\mapsto H[t,x(t)]$ is a preserving orientation diffeomorphism.

Notice that the continuity of $\partial \phi(s,t,x(t))/\partial x(t)$ for any $s\leq t$ (ensured by Theorem 4.1 from \cite[Ch.V]{Hartman}) implies 
the continuity of $\partial H[t,x(t)]/\partial x(t)$ and we conclude that $H[t,x(t)]$ is a diffeomorphism.

The Liouville's formula (see \emph{e.g.}, Theorems 7.2 and 7.3 from \cite[Ch.1]{Coddington}) says that 
$$
\det\Psi(t,s)>0 \quad \textnormal{and} \quad \det\frac{\partial \phi(s,t,x(t))}{\partial x(t)}>0  \quad
\textnormal{for any} \quad s\leq t
$$
and \textbf{(D2)} implies that these inequalities are preserved at $s=-\infty$, and we conclude
that $x(t)\mapsto H[t,x(t)]$ is a preserving orientation diffeomorphism.
\end{proof}

\begin{remark}
As $t\mapsto H[t,x(t)]$ is solution of (\ref{lin}), the uniqueness of the solution implies that
\begin{equation}
\label{TH}
H[t,\phi(t,0,\xi)]=\Psi(t,0)H[0,\xi]. 
\end{equation}
 \end{remark}

\begin{remark}
\label{equiv-asympt}
The matrix differential equation (\ref{perturbacion1}) can be seen as a perturbation of the matrix equation
\begin{equation}
\label{Lin-one}
\left\{\begin{array}{rcl}
X'(s)&=&A(s)X(s)\\
X(t)&=&I
\end{array}\right.
\end{equation}
related to (\ref{lin}). In addition, (\ref{Lin-one}) has a solution $s\mapsto X(s)=\Psi(s,t)$. 

Notice that $\Psi(t,s)X(s)=I$ while Theorem \ref{anexo-palmer} says
that $s\mapsto \Psi(t,s)Y(s)$ exists at $s=-\infty$. This fact prompt us that the behavior of
(\ref{perturbacion1}) and (\ref{Lin-one}) at $s\to -\infty$ has some relation weaker than asymptotic equivalence. Indeed, 
in \cite{Akhmet},\cite{Rab} it is proved that \textbf{(D1)} is a necessary condition for asymptotic equivalence between a linear system
and a linear perturbation. 
\end{remark}

\begin{theorem}
\label{SegDer} 
If \textnormal{\textbf{(H1)--(H4)}} and \textnormal{\textbf{(D1)--(D3)}} are satisfied, then, for 
any fixed $t$, the function $x\mapsto H(t,x)$ is a $C^{2}$ preserving orientation diffeomorphism. In particular, if 
$t\mapsto x(t)$ is a solution of \textnormal{(\ref{no-lin})}, then, for any fixed $t$,
$x(t)\mapsto H[t,x(t)]$ is a $C^{2}$ preserving orientation diffeomorphism.
\end{theorem}

\begin{proof}
Let us denote $x(t)=\big(x_{1}(t),\ldots,x_{n}(t)\big)=\phi(t,0,\xi)$. As in the previous result, the proof will 
be decomposed in several steps:\\ 

\noindent \textit{Step 1:} About $\partial^{2}H[t,x(t)]/\partial x_{j}(t)\partial x_{i}(t)$.\\
\noindent For any $i,j\in \{1,\ldots,n\}$, we can verify that
\begin{displaymath}
\begin{array}{rcl}
\displaystyle \frac{\partial^{2}H}{\partial x_{j}\partial x_{i}}[t,x(t)]&=& \displaystyle -\int_{-\infty}^{t}\hspace{-0.3cm}\Psi(t,s)D^{2}f(s,\phi(s,t,x(t)))
\frac{\partial \phi(s,t,x(t))}{\partial x_{j}}\frac{\partial \phi(s,t,x(t))}{\partial x_{i}}\,ds\\\\
&&\displaystyle -\int_{-\infty}^{t}\Psi(t,s)Df(s,\phi(s,t,x(t)))\frac{\partial^{2}\phi(s,t,x(t))}{\partial x_{j}\partial x_{i}}\,ds,
\end{array}
\end{displaymath}
where $x_{i}=x_{i}(t)$ and $x_{j}=x_{j}(t)$.

Now, by using (\ref{MDE15}) and (\ref{MDE2}), the reader can verify that
\begin{displaymath}
\begin{array}{rcl}
 \displaystyle  \frac{\partial^{2}H}{\partial x_{j}\partial x_{i}}[t,x(t)] &=&\displaystyle -\int_{-\infty}^{t}\frac{d}{ds}\Big\{\Psi(t,s)\frac{\partial^{2}\phi(s,t,x(t))}{\partial x_{j}\partial x_{i}}\Big\}\,ds\\\\
                                &=&\displaystyle \lim\limits_{s\to -\infty}\Psi(t,s)\frac{\partial^{2}\phi(s,t,x(t))}{\partial x_{j}\partial x_{i}}
\end{array}
\end{displaymath}
and the existence of $\partial^{2}H[t,x(t)]/\partial x_{j}(t)\partial x_{i}(t)$ follows if and only if the limit above exists.

\noindent\textit{Step 2:}  $\partial^{2}H[t,x(t)]/\partial x_{j}(t)\partial x_{i}(t)$ is well defined.\\

By using (\ref{MDE1}) and (\ref{MDE2}), we can see that $s\mapsto \Psi(t,s)\partial\phi(s,t,x(t))/\partial x_{i}$ is a
solution of (\ref{L2}) passing through $e_{i}$ at $s=t$. In consequence, we can deduce that
$$
\Psi(t,s)\frac{\partial \phi(s,t,x(t))}{\partial x_{i}}=Z(s,x(t))e_{i}
$$
and
$$
\Psi(t,s)\frac{\partial^{2} \phi(s,t,x(t))}{\partial x_{j}\partial x_{i}}=\frac{\partial Z}{\partial x_{j}}(s,x(t))e_{i}.
$$

By \textbf{(D3)}, the last identity has limit when $s\to -\infty$ and 
$\partial^{2}H[t,x(t)]/\partial x_{j}\partial x_{i}$ is well defined and continuous with respect to
$x(t)$.
\end{proof}

\begin{remark}
A careful reading of the results above, shows that our methods can be generalized in order to prove that
$H$ is a $C^{r}$ diffeomorphism with $r\geq 2$.
\end{remark}



\section{Density function}
As we pointed out in subsection 1.2, \textbf{(H1)} implies that (\ref{lin})
is uniformly asymptotically stable, which is a particular case of global asymptotical stability.
Now, by Proposition \ref{dens-lin}, there exists a density function 
$\rho\in C^{1}(\mathbb{R}\times \mathbb{R}^{n}\setminus\{0\},[0,+\infty))$ associated
to (\ref{lin}). By following the ideas for the autonomous case studied by Monz\'on \cite[Prop. III.1]{Monzon-0}
combined with the function $\rho$, we state the following result:
\begin{theorem}
\label{ex-dens0}
If \textnormal{\textbf{(H1)--(H5)}} and \textnormal{\textbf{(D1)--(D3)}} are satisfied,
then there exists a density function $\bar{\rho}\in 
C(\mathbb{R}\times \mathbb{R}^{n}\setminus\{0\},[0,+\infty))$ 
associated to \textnormal{(\ref{no-lin})}, defined by
\begin{equation}
\label{dens-nl}
\bar{\rho}(t,x)=\rho(t,H(t,x))\Big|\frac{\partial H(t,x)}{\partial x}\Big|,
\end{equation}
where $H(\cdot,\cdot)$ is the $C^{2}$ preserving orientation diffeomorphism defined 
before, $x$ is any initial condition of \textnormal{(\ref{no-lin})} and
$|\cdot|$ denotes a determinant. 
\end{theorem}

\begin{proof}
In spite that in the previous sections, the initial condition and the determinant were respectively denoted
by $\xi$ and $\det(\cdot)$, the notation of (\ref{dens-nl}) is classical in the density function literature. The reader will not be
disturbed by this fact.

We shall prove that (\ref{dens-nl}) satisfies the properties of Definition \ref{density}
with $g(t,x)=A(t)x+f(t,x)$. Indeed, $\bar{\rho}$ is non--negative since $\rho$ is non--negative
and $H$ is preserving orientation. In addition, $\bar{\rho}$ is $C^{1}$ since $H$ is $C^{2}$.

The rest of the proof will be decomposed in several steps:

\noindent\emph{Step 1:} $\bar{\rho}(t,x)$ is integrable outside any ball centered in the origin.

Let $B$ be an open ball centered at the origin. By using $H(t,0)=0$ and statement (i) from Proposition \ref{difeo}, we can 
conclude that $H(t,B)$ is an open and bounded set containing the origin. In consequence, for any fixed $t$, the outside of $B$ is mapped in the 
outside of another ball centered at the origin and contained in $H(t,B)$.

Let $\mathcal{Z}$ be a measurable set whose closure does not contain the origin. The property stated above implies that
$H(t,\mathcal{Z})$ is outside of some ball centered at the origin. Now, by the change of variables
theorem, we can see that
$$
\displaystyle \int_{\mathcal{Z}}\bar{\rho}(t,x)\,dx
=
\int_{\mathcal{Z}}\rho(t,H(t,x))\Big|\frac{\partial H(t,x)}{\partial x}\Big|\,dx
=
\int_{H(t,\mathcal{Z})}\rho(t,y)\,dy.
$$

Finally, as $\rho(t,\cdot)$ is integrable outside any open ball centered at the origin, the same follows for $\bar{\rho}(t,\cdot)$.

\noindent\emph{Step 2:} $\bar{\rho}(t,x)$ verifies
\begin{equation}
\label{postividad}
\frac{\partial \bar{\rho}}{\partial t}(t,x)+\triangledown\cdot (\bar{\rho}g)(t,x)>0 \quad 
\textnormal{a.e. in $\mathbb{R}^{n}$}. 
\end{equation}

Firstly, by using the Liouville's formula (see \emph{e.g.}, \cite[Corollary 3.1]{Hartman}), we know that
\begin{displaymath}
\frac{\partial}{\partial \eta}\Big|\frac{\partial \phi(\tau+t,t,x)}{\partial x}\Big|\Bigg|_{\tau=0}=\triangledown \cdot g(t,x),
\end{displaymath}
where $\eta=\tau+t$. Now, it is easy to verify that:
\begin{displaymath}
\begin{array}{rcl}
\displaystyle \frac{\partial \bar{\rho}}{\partial t}(t,x)+\triangledown \cdot (\bar{\rho}g)(t,x)&=&
\displaystyle \frac{\partial}{\partial \eta}\Big\{\bar{\rho}(\tau+t,\phi(\tau+t,t,x))\displaystyle \Big|\frac{\partial \phi(\tau+t,t,x)}{\partial x}\Big|\Big\}\Big|_{\tau=0}\\\\
&=&\displaystyle \frac{\partial}{\partial \eta}\Big\{\rho(\tau+t,H[\tau+t,\phi(\tau+t,t,x)])\\\\\
& & \displaystyle \Big|\frac{\partial H[\tau+t,\phi(\tau+t,t,x)]}{\partial \phi(\tau+t,t,x)}\Big|\Big|\frac{\partial \phi(\tau+t,t,x)}{\partial x}\Big|\Big\}\Big|_{\tau=0}\\\\
&=& \displaystyle \frac{\partial}{\partial \eta}\Big\{\rho(\tau+t,H[\tau+t,\phi(\tau+t,t,x)])\\\\
& & \displaystyle \Big|\frac{\partial H[\tau+t,\phi(\tau+t,t,x)]}{\partial x}\Big|\Big\}\Big|_{\tau=0}.\\\\
\end{array}
\end{displaymath}

Secondly, a consequence of (\ref{TH}) is
\begin{displaymath}
H[\tau+t,\phi(\tau+t,t,x)]=\Psi(\tau+t,t)H(t,x),
\end{displaymath}
which implies:
\begin{displaymath}
\begin{array}{rcl}
\displaystyle \frac{\partial \bar{\rho}}{\partial t}(t,x)+\triangledown \cdot (\bar{\rho}g)(t,x) &=& \displaystyle  \frac{\partial}{\partial \eta}\Big\{\rho(\tau+t,\Psi(\tau+t,t)H(t,x))
\displaystyle \Big|\frac{\partial \Psi(\tau+t,t)H(t,x)}{\partial x}\Big|\Big\}\Big|_{\tau=0}\\\\
&=&A_{1}(\tau+t,x)+A_{2}(\tau+t,x)\Big|_{\tau=0},
\end{array}
\end{displaymath}
where $A_{1}(\cdot,\cdot)$ and $A_{2}(\cdot,\cdot)$ are respectively defined by
\begin{displaymath}
\begin{array}{rcl}
A_{1}(\tau+t,x)&=&\displaystyle \frac{\partial }{\partial \eta}\Big\{\rho(\tau+t,\Psi(\tau+t,t)H(t,x))\Big\}\Big|\frac{\partial \Psi(\tau+t,t)H(t,x)}{\partial x}\Big|\\\\
&=&\displaystyle \Big\{\frac{\partial \rho}{\partial \eta}(\tau+t,\Psi(\tau+t,t)H(t,x))+\\\\
& &\triangledown\rho\Big(\tau+t,\Psi(\tau+t,t)H(t,x)\Big)A(\tau+t)\Psi(\tau+t,t)H(t,x)\Big\}\\\\
& &\displaystyle \Big|\frac{\partial \Psi(\tau+t,t)H(t,x)}{\partial x}\Big|
\end{array}
\end{displaymath}
and
\begin{displaymath}
\begin{array}{rcl}
A_{2}(\tau+t,x)&=&\rho(\tau+t,\Psi(\tau+t,t)H(t,x))\displaystyle \frac{\partial}{\partial \eta}
\Big\{\Big|\frac{\partial \Psi(\tau+t,t)H(t,x)}{\partial x}\Big|\Big\}\\\\
&=&\rho(\tau+t,\Psi(\tau+t,t)H(t,x))\\\\
& &\displaystyle \frac{\partial}{\partial \eta}\Big\{\Big|\frac{\partial \Psi(\tau+t,t)H(t,x)}{\partial H(t,x)}\Big|\Big|\frac{\partial H(t,x)}{\partial x}\Big|\Big\}
\end{array}
\end{displaymath}

As
\begin{displaymath}
\begin{array}{rcl}
\displaystyle A_{1}(t,x)&=&\displaystyle \Big\{\frac{\partial \rho}{\partial \eta}(t,H(t,x))+\triangledown\rho(t,H(t,x))A(t)H(t,x)\Big\}\Big|\frac{\partial H(t,x)}{\partial x}\Big|
\end{array}
\end{displaymath}
and
\begin{displaymath}
\begin{array}{rcl}
A_{2}(t,x)&=&\displaystyle \rho(t,H(t,x))\tr A(t)H(t,x)\Big|\frac{\partial H(t,x)}{\partial x}\Big|,
\end{array}
\end{displaymath}
we can conclude that

\begin{displaymath}
\begin{array}{rcl}
\displaystyle \frac{\partial \bar{\rho}}{\partial t}(t,x)+\triangledown \cdot (\bar{\rho}g)(t,x)&=&A_{1}(t,x)+A_{2}(t,x)\\\\
&=&\displaystyle \Big\{\frac{\partial \rho}{\partial \eta}(t,H(t,x))+\triangledown \cdot \rho(t,H(t,x)) A(t)H(t,x)\Big\}\Big|\frac{\partial H(t,x)}{\partial x}\Big|,
\end{array}
\end{displaymath}
which is positive since is the product of two positive terms. The positiveness of the first one is ensured by Proposition \ref{dens-lin}, 
while the second follows by Theorem \ref{anexo-palmer}.

\noindent\emph{Step 3:} End of proof.

As we commented before, the existence of density function associated to (\ref{no-lin}) is based on the homeomorphism $H$ constructed by 
Palmer (Proposition \ref{difeo}) and the existence of the density function $\rho(t,x)$ associated to (\ref{lin}) constructed by 
Monz\'on (Proposition \ref{dens-lin}). Proposition \ref{anexo-palmer} and Theorem \ref{SegDer} ensure that $H$ is a $C^{2}$ preserving orientation 
diffeomorphism while the previous steps state that (\ref{dens-nl}) is indeed a density function associated to (\ref{no-lin}) and the result follows.
\end{proof}

\subsection{An application to nonlinear systems}
Let us consider the nonlinear system
\begin{equation}
\label{dugma}
x'=g(t,x) 
\end{equation}
where $g$ is a $C^{2}$ function satisfying
\begin{itemize}
\item[\textbf{(H1')}] $g(t,0)=0$ and $|g(t,x)|\leq \tilde{\mu}$ for any $t\in \mathbb{R}$ and $x\in \mathbb{R}^{n}$.
\item[\textbf{(H2')}] $|g(t,x_{1})-g(t,x_{2})|\leq L |x_{1}-x_{2}|$ for any $t\in \mathbb{R}$.
\end{itemize}

\begin{corollary}
\label{application}
If:
\begin{itemize}
\item[\textbf{(G1)}] The linear system $y'=Dg(t,0)y$ is exponentially stable and its transition matrix satisfy 
$$
||\Phi(t,s)||\leq Ke^{-\alpha(t-s)} \quad \textnormal{for some} \quad K\geq 1 \quad \textnormal{and} \quad \alpha>0.
$$
\item[\textbf{(G2)}] The Lipschitz constant $L$ satisfies
$$
L+||Dg(t,0)|| \leq \alpha/4K \quad \textnormal{for any} \quad t\in \mathbb{R},
$$
\item[\textbf{(G3)}] The first derivative of $g$ is such that
\begin{displaymath}
\int_{-\infty}^{t}||\widetilde{F}(r,\xi)||_{\infty}\,dr< 1 
\end{displaymath}
for any fixed $t$, with
\begin{displaymath}
\widetilde{F}(r,\xi)=\Phi(t,r)\{Dg(r,\varphi(r,0,\xi))-Dg(r,0)\}\Phi(r,t),
\end{displaymath}
where $\varphi(r,0,\xi)$ is the solution of \textnormal{(\ref{dugma})} passing through $\xi$ at $r=0$.
\item[\textbf{(G4)}] For any fixed $t$, $Dg(t,0)$ and $Dg(t,\varphi(t,0,\xi))$ are such that
\begin{displaymath}
\liminf\limits_{s\to-\infty}-\int_{s}^{t}\tr Dg(r,0)\,dr>-\infty \hspace{0.2cm} \textnormal{and} \hspace{0.2cm}
\liminf\limits_{s\to-\infty}-\int_{s}^{t}\tr Dg(r,\varphi(r,0,\xi))\,dr>-\infty
\end{displaymath}
for any initial condition $\xi$.
\item[\textbf{(G5)}] For any fixed $t$ and $i,j=1,\ldots,n$, the following limit exists
\begin{displaymath}
\lim\limits_{s\to -\infty}\frac{\widetilde{Z}(s,x(t))}{\partial x_{j}(t)}e_{i},
\end{displaymath}
where $x(t)=\varphi(t,0,\xi)$ and $\widetilde{Z}(s,x(t))$ is a fundamental matrix of 
\begin{displaymath}
\widetilde{Z}'=\widetilde{F}(s,x(t))\widetilde{Z}.
\end{displaymath}
\end{itemize}
then there exists a density function $\bar{\rho}\in 
C(\mathbb{R}\times \mathbb{R}^{n}\setminus\{0\},[0,+\infty))$ associated to \textnormal{(\ref{dugma})}.
\end{corollary}

\section{Illustrative Example}
Let us consider the scalar equation
\begin{equation}
\label{baby}
x'=-ax+h(t)\arctan(x), 
\end{equation}
where $a>0$ and $h\colon \mathbb{R}\to \mathbb{R}_{+}$ is bounded and continuous. In addition, we will suppose that 
\begin{equation}
\label{baby1}
r\mapsto h(r)e^{-ar} \quad \textnormal{is integrable on} \quad (-\infty,\infty). 
\end{equation}

It is easy to see that \textbf{(H1)--(H2)} are satisfied with $\mu=||h||_{\infty}\pi/2$
and $\gamma=||h||_{\infty}$.

Notice that that \textbf{(H3)} is satisfied since $\Psi(t,s)=e^{-a(t-s)}$ and \textbf{(H4)}
is satisfied if and only if $4||h||_{\infty}\leq a$.

Moreover, \textbf{(D1)} is satisfied if for any solution $r\mapsto \phi(r,0,\xi)$ of (\ref{baby}) 
\begin{equation}
\label{int-conver}
\int_{-\infty}^{\infty}\frac{h(r)}{1+\phi^{2}(r,0,\xi)}\,dr<1.
\end{equation}

It is interesting to point out $\phi(t,0,\xi)$ is unbounded and have exponential growth at $t=-\infty$. Now, it is easy to note that
$$
\lim\limits_{s\to -\infty} -as=+\infty,
$$
which implies that
$$
\liminf\limits_{s\to -\infty}-\int_{s}^{t}\Big\{-a+\frac{h(r)}{1+\phi^{2}(r,0,\xi)}\Big\}\,dr>-\infty
$$
for any fixed $t$, and \textbf{(D2)} is satisfied.

Letting $f(t,x)=h(t)\arctan(x)$ 
and noticing that
$$
Z(s,x(t))=\exp\Big\{-\int_{s}^{t}Df(u,\phi(u,t,x(t)))\,du\Big\},
$$
and
$$
\frac{\partial \phi(s,t,x(t))}{\partial x(t)}=\exp\Big\{a(t-s)-\int_{s}^{t}Df(u,\phi(u,t,x(t)))\,du\Big\}
$$
with $x(t)=\phi(t,0,\xi)$. Consequently, a straigthforward computation shows that \textbf{(D3)} is satisfied if and only if
$$
Z(s,x(t))\Big[\int_{s}^{t}\exp\Big(a\{t-u\}-\int_{u}^{t}Df(r,\phi(r,t,x(t)))\,dr\Big)D^{2}f(u,\phi(u,t,x(t)))\,du\Big],
$$
has limit when $s\to -\infty$.

Finally, (\ref{baby1}) and (\ref{int-conver}) imply that \textbf{(D3)} is satisfied since the function
$$
u\mapsto \frac{h(u)e^{-a(u-t)}\phi(u,t,x(t))}{(1+\phi^{2}(u,t,x(t)))^{2}}=\frac{h(u)e^{-a(u-t)}\phi(u,0,\xi)}{(1+\phi^{2}(u,0,\xi))^{2}}
$$
is integrable on $(-\infty,t]$ for any $t\in \mathbb{R}$.

\end{document}